\documentclass[a4paper,12pt]{extarticle}
\usepackage{xcolor}
\usepackage[a4paper, total={7in, 8in}]{geometry}

\usepackage{cancel}
\usepackage[section]{placeins}
\usepackage{pgfplots}
\usepackage{amsmath}
\usepackage{mathrsfs}
\usetikzlibrary{arrows}
\usepackage{amsmath, amsthm, amsfonts,amssymb,stmaryrd}

\usepackage{pst-node} 
\usepackage[extra]{tipa}
\usepackage{pstricks}
\usepackage{tikz}
\usepackage[spanish,USenglish]{babel}      

\usepackage [T1]{fontenc}
\usepackage{amssymb,amsthm,amsmath,amsfonts}
\usepackage{graphicx}
\usepackage[affil-it]{authblk}
\usepackage{url}
\usepackage{subfigure}
\usepackage{comment}
\usepackage{tikz}
\usetikzlibrary{matrix, arrows.meta, positioning}

\usepackage{colortbl} 

\usepackage[colorlinks=true, allcolors=blue]{hyperref} 
\usetikzlibrary{shapes.geometric}
\usepackage{graphicx}
\usepackage{soul}
\usepackage{xcolor}

\newcommand{\reduline}[1]{\setulcolor{red}\ul{#1}}

\hypersetup{colorlinks=true, linkcolor=blue}

\usepackage{enumerate}
\usetikzlibrary{shapes.geometric}

\newtheorem{thm}{Theorem}[section]

\newtheorem{cor}[thm]{Corollary}
\newtheorem{lem}[thm]{Lemma}
\newtheorem{prop}[thm]{Proposition}
\theoremstyle{definition}
\newtheorem{notn}[thm]{Notation}

\theoremstyle{remark}
\newtheorem{rem}[thm]{Remark}

\title { Maximum  augmented Zagreb index on  
polyomino 
chains.}
\author{Manuel Montes-y-Morales} 
\author{Saylé Sigarreta}
\author{Hugo Cruz-Suárez}
\date{\today}

\affil{Facultad de Ciencias Físico Matemáticas,\\
Benemérita Universidad Autónoma de Puebla, Puebla, Mexico}
\date{}
\begin{document}

\maketitle
\abstract{
In this paper, we present a dynamic programming approach for identifying  extremal polyomino chains with respect to degree-based topological indices. This approach provides an explicit recurrence and constructive algorithm that enables both the computation of an extremal  polyomino chain in linear time with respect to its number of squares, and the enumeration of all maximal configurations in linear time with respect to their amount. As a main application, we resolve a problem posed in 2016 by characterizing the polyomino chains that maximize the Augmented Zagreb Index ($AZI$) for any fixed number of squares. The $AZI$, a degree-based index known for its strong chemical applicability in numerous studies, attains its maximum on two specific families of polyomino chains depending on the parity of their number of squares. We also derive closed-form expressions for the maximum $AZI$ and determine the exact number of extremal configurations. The results presented in this paper are aligned with previous contributions, and establish a constructive methodology for solving extremal problems in chemical graph theory, for which we provide a link to the code in the last section}


\section{Introduction}

Graph theory offers a powerful approach for exploring the relationship between the molecular structure of a material and its physicochemical properties. In this context, molecules are represented as graphs, where atoms correspond to vertices and chemical bonds to edges. These molecular graphs are analyzed using numerical descriptors known as topological indices \cite{a8}. Among the various classes of topological indices, degree-based indices play a particularly significant role due to their simplicity and strong predictive power in chemical applications.

These indices are especially useful in chemical graph theory. For instance, the geometric-arithmetic ($GA$) index has been successfully employed as a predictive tool in QSPR/QSAR studies \cite{a11}. Similarly, the atom-bond connectivity ($ABC$) index has proven valuable in modeling molecular properties such as polarity and refractivity \cite{a12}. Furthermore, the Randi\'c ($R$) index has shown strong correlations with the solubility of alkanes in water \cite{a10}. 

Among these descriptors, the Augmented Zagreb Index ($AZI$), was proposed as a natural extension of the classical Zagreb index to enhance their discriminatory power in chemical graph theory \cite{a2, a3}. The empirical effectiveness of the $AZI$ has been demonstrated in several studies, where it consistently outperforms other degree-based topological indices in modeling physicochemical properties. In particular, the $AZI$ shows stronger correlation coefficients with standard heats of formation and boiling points, especially for structurally similar compounds such as heptane and octane isomers \cite{a8}. These results highlight the ability of the $AZI$ to capture structural variations and its consistent performance across families of isomers. Despite significant progress in the study of $AZI$ across various graph classes, the problem of identifying the polyomino chains that maximize it, originally posed in 2016 \cite{a2}, has remained unresolved.

Polyomino systems are planar structures formed by connecting squares \cite{a17}. These configurations are particularly relevant in chemical graph theory as they can be used to model a variety of molecular structures, including polymers, crystal lattices, and certain classes of organic compounds \cite{a13}. In a polyomino system, two squares are said to be adjacent if they share a side. A polyomino system in which every square is adjacent with at most two other squares is called polyomino chain. Polyomino chains have been widely studied in the context of degree-based topological indices \cite{a18,a19}. Their recursive structure allows for a combinatorial encoding through link types, which facilitates algorithmic enumeration and structural analysis. Considerable attention has been devoted to characterizing extremal polyomino chains with respect to degree-based topological indices \cite{a7,aa6}.

In this work, we develop a dynamic programming approach to identify, for any fixed number of squares, the polyomino chains that maximize the $AZI$. Dynamic programming is a classical optimization technique that solves complex problems by breaking them into overlapping subproblems, each of which is solved only once and stored for future reuse \cite{a14,a15}. This method is particularly well suited for problems with recursive structure and optimal substructure properties, both characteristic features of the enumeration and optimization of polyomino chains. Our approach provides an explicit recurrence relation and a constructive algorithm that not only computes the extremal $AZI$ value, but also generates all maximal configurations. We prove that the maximum $AZI$ is attained by two specific families of polyomino chains, depending on whether the number of squares is even or odd. Furthermore, we derive closed-form expressions for the maximum $AZI$ and determine the exact number of maximal chains in each case.

The approach introduced herein establishes a general methodology for formulating and solving extremal problems involving degree-based topological indices in graph families with recursive structures. In doing so, it contributes to the ongoing effort to characterize extremal behavior in chemical graph theory.

\section{Preliminary Results}\label{s1}

This section presents the background and notation used throughout the paper. It begins with an introduction to the degree-based topological indices that form the foundation of the analysis. Next, the combinatorial structure of polyomino chains is described, including their representation through link types. Finally, a key lemma is provided, establishing a recursive formula for computing any degree-based topological index on polyomino chains.

The analysis in this paper centers on a particular family of topological indices known as \emph{degree-based indices}. These indices are defined by a general formula of the form:

\begin{equation}\label{TI}
   TI_f(G)= \sum_{uv\in E(G)} f(d_u, d_v),
\end{equation}
where $G$ is a graph, $E(G)$ is the set of edges of $G$, $d_u$ and $d_v$ are the degrees of the vertices $u$ and $v$, respectively, and $f$ is a real-valued symmetric function, that is, $f(x, y) = f(y, x)$ for all $x, y \in \{1, 2, \dots\}$.

This work focuses on the study of degree-based topological indices within a particular class of graphs called \emph{polyomino chains}.  A polyomino chain is a planar system composed of squares, where each square is adjacent to at most two others, i.e. it shares a side. It is worth noting that a polyomino chain consisting of $n$ squares, denoted by $PC_n$, can be generated recursively. Figure \ref{f1i} illustrates the base case $PC_2$. For $n \geq 2$, a new square can be added in one of two ways, as shown in Figure \ref{f2i}:
\begin{enumerate}
    \item [a)] Link type 1, in which the direction of the chain is preserved.
    \item [b)] Link type 2, in which the chain changes direction.
\end{enumerate}

This recursive process naturally induces a sequence composed of 1s and 2s, representing the link types used at each step. If we denote by $L_n$ the selected link type at stage $n \geq 3$, then a polyomino chain with $n$ squares can be represented by the notation $PC(L_3, L_4, \dots, L_n)$, where the direction of each added square follows the vector $(L_3, L_4, \dots, L_n)$. Two polyomino chains are considered distinct if their associated link vectors differ.
\begin{figure}[h!]

 \centering
\begin{tikzpicture}

\draw (3,0) rectangle (4,1);

\draw (4,0) rectangle (5,1);

(n4) at (1,1) {};
\node[fill=black, draw, circle, inner sep=2pt] (n5) at (3,0) {};
\node[fill=black, draw, circle, inner sep=2pt] (n6) at (3,1) {};
\node[fill=black, draw, circle, inner sep=2pt] (n7) at (4,0) {};
\node[fill=black, draw, circle, inner sep=2pt] (n8) at (4,1) {};
\node[fill=black, draw, circle, inner sep=2pt] (n7) at (5,1) {};
\node[fill=black, draw, circle, inner sep=2pt] (n8) at (5,0) {};

\end{tikzpicture}

    \caption{The graph $PC_2$.}

    \label{f1i}
\end{figure}
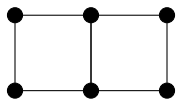

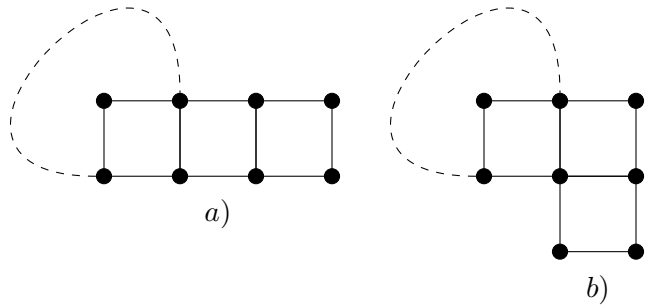
\begin{figure}[h!]
 \centering
\begin{tikzpicture}
  \clip (-3,-2) rectangle (9,3);

\draw (0,0) rectangle (1,1);

\draw (1,0) rectangle (2,1);

\draw (2,0) rectangle (3,1);

\node[fill=black, draw, circle, inner sep=2pt] (n1) at (0,0) {};
\node[fill=black, draw, circle, inner sep=2pt] (n2) at (1,0) {};
\node[fill=black, draw, circle, inner sep=2pt] (n3) at (0,1) {};
\node[fill=black, draw, circle, inner sep=2pt] (n4) at (1,1) {};
\node[fill=black, draw, circle, inner sep=2pt] (n5) at (2,0) {};
\node[fill=black, draw, circle, inner sep=2pt] (n6) at (3,0) {};
\node[fill=black, draw, circle, inner sep=2pt] (n7) at (2,1) {};
\node[fill=black, draw, circle, inner sep=2pt] (n8) at (3,1) {};
  \draw[ dashed] (n1) to  [out=180, in=90,looseness=5] (n4);

\draw (5,0) rectangle (6,1);
\draw (6,0) rectangle (7,1);
\draw (6,-1) rectangle (7,0);

\node[fill=black, draw, circle, inner sep=2pt] (n9) at (5,0) {};
\node[fill=black, draw, circle, inner sep=2pt] (n10) at (6,0) {};
\node[fill=black, draw, circle, inner sep=2pt] (n11) at (7,0) {};
\node[fill=black, draw, circle, inner sep=2pt] (n12) at (5,1) {};
\node[fill=black, draw, circle, inner sep=2pt] (n13) at (6,1) {};
\node[fill=black, draw, circle, inner sep=2pt] (n14) at (7,1) {};
\node[fill=black, draw, circle, inner sep=2pt] (n15) at (6,-1) {};
\node[fill=black, draw, circle, inner sep=2pt] (n16) at (7,-1) {};
\draw[ dashed] (n9) to  [out=180, in=90,looseness=5] (n13);

\node at (1.5,-0.5) {\small $a)$};
\node at (6.5,-1.5) {\small $b)$};
\end{tikzpicture}
   \caption{The two link ways for $PC_n$.}
    \label{f2i}
\end{figure}

\begin{rem}
In line with the definition of a polyomino chain and the preceding construction, this paper restricts the study of degree-based indices on polyomino chains to those formed by successively attaching each new square either to the right or below the previous one.  However, the family of polyomino chains is much broader, as illustrated in Figure~\ref{fn1}. Therefore, in what follows, we shall adopt the convention that: the term \emph{polyomino chains} will be reserved for the \emph{restricted} case studied in this paper, while the term \emph{general polyomino chains} will refer to those satisfying the definition without additional constraints.

\begin{figure}[h!]

 \centering
\begin{tikzpicture}

\draw (0,0) rectangle (1,1);
\draw (1,0) rectangle (2,1);
\draw (1,-1) rectangle (2,0);
\draw (1,-2) rectangle (2,-1);
\draw (0,-2) rectangle (1,-1);

\node[fill=black, draw, circle, inner sep=2pt] (n5) at (0,0) {};
\node[fill=black, draw, circle, inner sep=2pt] (n5) at (0,1) {};
\node[fill=black, draw, circle, inner sep=2pt] (n6) at (1,0) {};
\node[fill=black, draw, circle, inner sep=2pt] (n7) at (1,1) {};
\node[fill=black, draw, circle, inner sep=2pt] (n8) at (1,-1) {};
\node[fill=black, draw, circle, inner sep=2pt] (n7) at (1,-2) {};
\node[fill=black, draw, circle, inner sep=2pt] (n8) at (0,-2) {};

\node[fill=black, draw, circle, inner sep=2pt] (n9) at (0,-1) {};
\node[fill=black, draw, circle, inner sep=2pt] (n10) at (2,1) {};
\node[fill=black, draw, circle, inner sep=2pt] (n11) at (2,0) {};
\node[fill=black, draw, circle, inner sep=2pt] (n12) at (2,-1) {};
\node[fill=black, draw, circle, inner sep=2pt] (n12) at (2,-2) {};

\node at (1,-2.5) {$PC_5$};

\end{tikzpicture}

    \caption{The eight polyomino chains consisting of five squares, constructed under the constraint that each new square is appended either to the right or below the preceding square, are not isomorphic to $PC_5$.
}

    \label{fn1}
\end{figure}
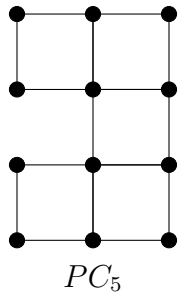

\end{rem}
A \emph{linear chain}, denoted by $Li_n$, corresponds to the case where all links are of type 1. Conversely, a \emph{zig-zag chain}, denoted by $Z_n$, is formed when all links are of type 2 (see Figure~\ref{f3i}). In this framework, a square is called \emph{terminal} if it has exactly one adjacent square, \emph{medial} if it has two adjacent squares and contains no vertex of degree 2; and a \emph{kink} if it has two adjacent squares and contains a vertex of degree 2. A \emph{segment} $s$ is defined as a maximal linear chain together with an adjacent kink or terminal square.  A segment is said to be \emph{external} if it contains a terminal square; otherwise, it is referred to as an \emph{internal} segment. Thus, every polyomino chain $PC_n$ induces a sequence of segments $s_1, s_2, \dots, s_m$ (see Figure~\ref{f5}), where the length of a segment $s_i$, defined as its number of squares, is denoted by $l(s_i) = l_i$. These lengths satisfy the relation: $\sum_{i=1}^{m} l_i = n + m - 1$.

 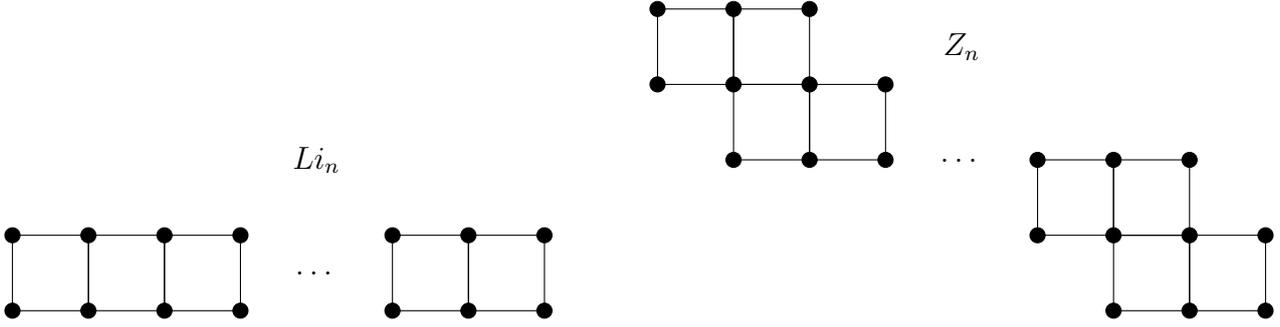
\begin{figure}
\centering

\begin{tikzpicture}


\draw (0,0) rectangle (1,1);

\draw (1,0) rectangle (2,1);

\draw (2,0) rectangle (3,1);

\node[fill=black, draw, circle, inner sep=2pt] (n1) at (0,0) {};
\node[fill=black, draw, circle, inner sep=2pt] (n2) at (1,0) {};
\node[fill=black, draw, circle, inner sep=2pt] (n3) at (0,1) {};
\node[fill=black, draw, circle, inner sep=2pt] (n4) at (1,1) {};
\node[fill=black, draw, circle, inner sep=2pt] (n5) at (2,0) {};
\node[fill=black, draw, circle, inner sep=2pt] (n6) at (3,0) {};
\node[fill=black, draw, circle, inner sep=2pt] (n7) at (2,1) {};
\node[fill=black, draw, circle, inner sep=2pt] (n8) at (3,1) {};

\node at (4,0.5) {$\dots$};
\node at (4,2) {$Li_n$};
\draw (5,0) rectangle (6,1);
\draw (6,0) rectangle (7,1);

\node[fill=black, draw, circle, inner sep=2pt] (n9) at (5,0) {};
\node[fill=black, draw, circle, inner sep=2pt] (n10) at (6,0) {};
\node[fill=black, draw, circle, inner sep=2pt] (n11) at (7,0) {};
\node[fill=black, draw, circle, inner sep=2pt] (n12) at (5,1) {};
\node[fill=black, draw, circle, inner sep=2pt] (n13) at (6,1) {};
\node[fill=black, draw, circle, inner sep=2pt] (n14) at (7,1) {};

\end{tikzpicture}
\hspace{1cm} 
\begin{tikzpicture}

\draw (0,0) rectangle (1,1);

\draw (1,0) rectangle (2,1);

\node[fill=black, draw, circle, inner sep=2pt] (n1) at (0,0) {};
\node[fill=black, draw, circle, inner sep=2pt] (n2) at (1,0) {};
\node[fill=black, draw, circle, inner sep=2pt] (n3) at (0,1) {};
\node[fill=black, draw, circle, inner sep=2pt] (n4) at (1,1) {};
\node[fill=black, draw, circle, inner sep=2pt] (n5) at (2,0) {};
\node[fill=black, draw, circle, inner sep=2pt] (n6) at (2,1) {};

\draw (1,-1) rectangle (2,0);
\draw (2,-1) rectangle (3,0);

\node[fill=black, draw, circle, inner sep=2pt] (n7) at (1,-1) {};
\node[fill=black, draw, circle, inner sep=2pt] (n8) at (2,-1) {};
\node[fill=black, draw, circle, inner sep=2pt] (n9) at (3,-1) {};
\node[fill=black, draw, circle, inner sep=2pt] (n10) at (3,0) {};

\node at (4,-1) {$\dots$};
\node at (4,0.5) {$Z_n$};

\draw (5,-2) rectangle (6,-1);
\draw (6,-2) rectangle (7,-1);
\draw (6,-3) rectangle (7,-2);
\draw (7,-3) rectangle (8,-2);

\node[fill=black, draw, circle, inner sep=2pt] (n11) at (5,-2) {};
\node[fill=black, draw, circle, inner sep=2pt] (n12) at (6,-1) {};
\node[fill=black, draw, circle, inner sep=2pt] (n13) at (6,-2) {};
\node[fill=black, draw, circle, inner sep=2pt] (n14) at (7,-1) {};
\node[fill=black, draw, circle, inner sep=2pt] (n15) at (6,-3) {};
\node[fill=black, draw, circle, inner sep=2pt] (n16) at (7,-2) {};
\node[fill=black, draw, circle, inner sep=2pt] (n17) at (7,-3) {};
\node[fill=black, draw, circle, inner sep=2pt] (n18) at (8,-2) {};
\node[fill=black, draw, circle, inner sep=2pt] (n19) at (5,-1) {};
\node[fill=black, draw, circle, inner sep=2pt] (n20) at (8,-3) {};

\end{tikzpicture}

 \caption{The linear chain and the zigzag chain.}
    \label{f3i}

\end{figure}

 \begin{figure}[h!]
    \centering
   \begin{tikzpicture}

\draw (0,0) rectangle (1,1);

\draw (1,0) rectangle (2,1);
\draw (2,0) rectangle (3,1);
\draw (3,0) rectangle (4,1);
\draw (3,-1) rectangle (4,0);
\draw (3,-2) rectangle (4,-1);
\draw (4,-2) rectangle (5,-1);
\draw (4,-3) rectangle (5,-2);
\draw (5,-3) rectangle (6,-2);
\draw (6,-3) rectangle (7,-2);

\node[fill=black, draw, circle, inner sep=2pt] (n1) at (0,0) {};
\node[fill=black, draw, circle, inner sep=2pt] (n2) at (1,0) {};
\node[fill=black, draw, circle, inner sep=2pt] (n3) at (0,1) {};
\node[fill=black, draw, circle, inner sep=2pt] (n4) at (1,1) {};
\node[fill=black, draw, circle, inner sep=2pt] (n5) at (2,0) {};

\node[fill=black, draw, circle, inner sep=2pt] (n7) at (2,1) {};
\node[fill=black, draw, circle, inner sep=2pt] (n7) at (3,0) {};
\node[fill=black, draw, circle, inner sep=2pt] (n7) at (3,1) {};
\node[fill=black, draw, circle, inner sep=2pt] (n7) at (4,0) {};
\node[fill=black, draw, circle, inner sep=2pt] (n7) at (4,1) {};
\node[fill=black, draw, circle, inner sep=2pt] (n7) at (3,-1) {};
\node[fill=black, draw, circle, inner sep=2pt] (n7) at (4,-1) {};

 \node[fill=black, draw, circle, inner sep=2pt] (n7) at (2,1) {};
\node[fill=black, draw, circle, inner sep=2pt] (n7) at (5,-1) {};
\node[fill=black, draw, circle, inner sep=2pt] (n7) at (3,-2) {};
\node[fill=black, draw, circle, inner sep=2pt] (n7) at (4,-2) {};
\node[fill=black, draw, circle, inner sep=2pt] (n7) at (5,-2) {};
\node[fill=black, draw, circle, inner sep=2pt] (n7) at (6,-2) {};
\node[fill=black, draw, circle, inner sep=2pt] (n7) at (7,-2) {};

\node[fill=black, draw, circle, inner sep=2pt] (n7) at (4,-3) {};
\node[fill=black, draw, circle, inner sep=2pt] (n7) at (5,-3) {};
\node[fill=black, draw, circle, inner sep=2pt] (n7) at (6,-3) {};
\node[fill=black, draw, circle, inner sep=2pt] (n7) at (7,-3) {};

\node at (-0.2,0.5) { $s_1$};
\node at (3.5,1.2) { $s_2$};
\node at (2.8,-1.5) { $s_3$};
\node at (4.5,-0.8) {$s_4$};
\node at (3.8,-2.5) {$s_5$};

\end{tikzpicture}
   \caption{Segments of a polyomino chain.}
    \label{f5}
\end{figure}
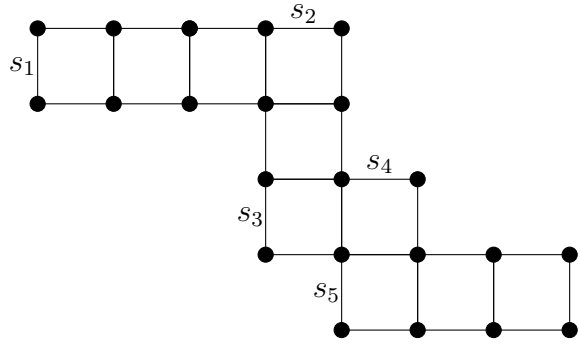
The following lemma summarizes a key result from \cite{a6}, which plays a central role in our work. We present it here with slight notational adjustments to align with our framework.

\begin{lem}\label{l1}
Let $ TI_f $ be a degree-based topological index, and let $ PC(L_3, \dots, L_n) $ be a polyomino chain with $ n \geq 3 $ squares. Then, the following recurrence holds:
$$
TI_f\big(PC(L_3, \dots, L_n)\big) = 
\begin{cases}
TI_f\big(PC(L_3, \dots, L_{n-1})\big) + g_f(L_{n-1}, L_n), & \text{if } n \geq 4, \\
TI_f\big(PC_2\big) + g_f(L_3), & \text{if } n = 3.
\end{cases}
,$$ where $PC_2$ is the polyomino chain with $2$ squares and \begin{itemize}
    
    \item[a)] $g_f(1,1)=g_f(1):=3f(3,3),$
    \item[b)] $g_f(1,2):=3f(3,4)+f(2,4)+f(2,3)-2f(3,3),$
    \item[c)] $g_f(2,1):=f(3,4)-f(2,4)+f(2,3)+2f(3,3),$
    \item[d)] $g_f(2,2):=f(4,4)+2f(2,4),$
    \item[e)] $g_f(2):=2f(3,4)+2f(2,4)-f(3,3).$
\end{itemize}
\end{lem}

\begin{rem}\label{r1}
    In particular, observe that $g_f(2)+g_f(2,1)=g_f(1,1)+g_f(1,2)$. This relationship will be used in the following sections.
\end{rem}
The recurrence provided in Lemma \ref{l1} shows that the contribution of each new square in a polyomino chain depends solely on the previous two links.  This observation allows us to apply a dynamic programming approach to identify extremal configurations with respect to degree-based topological indices.  

    \section{Dynamic Programming Approach applied to Polyomino chains} \label{s3}

Given a degree-based topological index $ TI_f $, we define $M_f(n,i)$ for $i = 1,2$ and $n\geq 3$, as the \emph{maximum value} of this index among all polyomino chains with $n$ squares and an ending link of type $i$; that is, $L_n = i$. In particular, we have $$M_f(3,1)=TI_f(PC(1))=TI_f(PC_2)+g_f(1,1),$$  and $$M_f(3,2)=TI_f(PC(2))=TI_f(PC_2)+g_f(2).$$ 
Motivated by the recursive nature of such indices in the polyomino chains, the following result provides a dynamic programming formulation for computing $M_f(n,i)$ for all $n\geq 4$ and $i\in \{1,2\}$.

\begin{thm}\label{t1}
Let $ TI_f $ be a degree-based topological index. Then, for all $ n \geq 4 $ and $i\in \{1,2\}$
$$
M_f(n,i) = \max\{ M_f(n-1,1) + g_f(1,i),\; M_f(n-1,2) + g_f(2,i) \}.
$$

\end{thm}

\begin{proof}
Given $j=1,2$, by definition of $M_f(n-1, j)$, there exist at least a polyomino chain $PC(L^{j}_3, \dots, L^{j}_{n-2},i)$, such that
$$
M_f(n-1,j) = TI_f(PC(L^{j}_3, \dots, L^{j}_{n-2},i) \big).$$ Hence, by Lemma \ref{l1}, we obtain 
$$M_f(n-1,1)+g_f(1,i)=TI_f\big( PC(L^{1}_3, \dots, L^{1}_{n-2},1,i) \big)$$ and

$$M_f(n-1,2)+g_f(2,i)=TI_f\big( PC(L^{2}_3, \dots, L^{2}_{n-2},2,i) \big).$$ 
Therefore,

$$ \max\{ M_f(n-1,1) + g_f(1,i),\; M_f(n-1,2) + g_f(2,i)\} \leq M_f(n,i).$$ Conversely, there exists at least one polyomino chain $ PC(L_3, \dots, L_{n-1}, i) $, such that

$$
M_f(n,i) = TI_f\big( PC(L_3, \dots, L_{n-1}, i)\big).
$$ Thus, by Lemma~\ref{l1}
\begin{align*}
M_f(n,i) &= TI_f\big( PC(L_3, \dots, L_{n-1})\big) + g_f(L_{n-1}, i) \\
         &\leq M_f(n-1, L_{n-1}) + g_f(L_{n-1}, i) \\
         &\leq \max\{ M_f(n-1,1) + g_f(1,i),\; M_f(n-1,2) + g_f(2,i) \}.
\end{align*}
This concludes the proof.
\end{proof}

Now, observe that due to the linearity of the degree-based topological indices (\ref{TI}), the identity $TI_{-f}(G)=-TI_{f}(G)$ holds. This implies that minimizing $TI_f$ is equivalent to maximizing $TI_{-f}$. Consequently, as a direct application of Theorem \ref{t1} and Lemma \ref{l1}, we obtain a method for computing the \emph{minimum value} of any degree-based topological index among all polyomino chains with $n$ squares and final link of type $i$, denoted by $m_f(n,i)$ with $n\geq3$ and $i\in \{1,2\}$.

\begin{cor}\label{c1}
   Let $ TI_f $ be a degree-based topological index. Then, for $ n \geq 3 $ and $i\in \{1,2\}$
$$
m_f(n,i) = -M_{-f}(n,i).
$$ 
\end{cor}
According to Corollary \ref{c1}, we may restrict our attention to the maximization problem. At this point, it is worth emphasizing that the method introduced in Theorem~\ref{t1} not only allows us to determine the maximum value of a fixed degree-based topological index over all polyomino chains with $n \geq 4$ squares and a final link of type $i$, but also provides a constructive way to generate at least one maximal polyomino chain. This construction can be performed recursively by selecting each $ L_j $, for $ j \in \{3, \dots, n-1\} $, according to the following rule: assign $ L_j = 1 $ if and only if
$$
M_f(j,1) + g_f(1, L_{j+1}) > M_f(j,2) + g_f(2, L_{j+1}),
$$
and assign $ L_j = 2 $ otherwise. Hereafter, we refer to this procedure as the \emph{constructive method} induced by the problem $ M_f(n,i) $. 
This naturally raises the following questions: What happens in the presence of ties? And how do these ties relate to the set of maximal polyomino chains? These issues are precisely addressed in the following results.

 From this point on, for fixed integers $ n \geq 3 $ and $ i \in \{1, 2\} $, we denote by $ n_t(n,i) $ the \emph{number of ties} that occur during the recursive computation of $M_f(n,i)$. These ties represent situations where both possible extensions of the chain yield the same maximum value, resulting in multiple equally optimal configurations. We define the base case as: $n_t(3,j) = 0$ for all  $j \in \{1,2\}$. For $ n \geq 4 $,  $n_t(n,i)$ is defined recursively as follows:
\begin{enumerate}
    \item [a)] if $M_f(n-1,j) + g_f(j,i) > M_f(n-1,3-j) + g_f(3-j,i)$, then $n_t(n,i)=n_t(n-1,j)$;
    \item [b)] if  $M_f(n-1,1) + g_f(1,i) = M_f(n-1,2) + g_f(2,i)$, then $$n_t(n,i)=1 + n_t(n-1,1) + n_t(n-1,2).$$
\end{enumerate}

\begin{lem}\label{l2}
   Let $n \geq 4$ and $i \in \{1,2\}$. Then the following statements hold:
\begin{enumerate}
    \item[a)] Suppose that $PC(L_3,\dots,L_{n-1},i)$ is a polyomino chain such that
   $$
    M_f(n,i) = TI_f\big(PC(L_3,\dots,L_{n-1},i)\big).
   $$
    Then, for all $j\in \{3,\dots,n-1\}$, it follows that
   $$
    M_f(j, L_j) = TI_f\big(PC(L_3, \dots, L_j)\big).
   $$
    
    \item[b)] Assume that $ PC(L^1_1, \dots, L^1_{n-1}, i) $ is a maximal polyomino chain corresponding to the problem $ M_f(n, i) $. Then, there exists at least one distinct maximal polyomino chain $ PC(L^2_1, \dots, L^2_{n-1}, i) $ if and only if
$$
n_t(j+1,L^1_{j+1})=1+n_t(j,1)+n_t(j,2),
$$
for some $j\in \{3,\dots,n-1\}$.

\item[c)] Given $j\in \{3,\dots,n-1\}$. The maximum number of consecutive links, counting backward from the last one, that are shared by all maximal polyomino chains of the problem  $M_f(n,i)$ is $ n - j $ and these links are precisely $ L_{j+1}, \dots, L_{n-1}, L_n=i $ if and only if 
$$
n_t(n,i) = n_t(n-1,L_{n-1}) = \dots = n_t(j+1,L_{j+1}) = 1 + n_t(j,1) + n_t(j,2).
$$

    \end{enumerate}
\end{lem}

\begin{proof}
\hfill

\begin{enumerate}

\item[a)] 

The first statement is immediate for the case $j=3$. For $j\geq 4$, assume by contradiction that
$$
TI_f\big( PC(L_3, \dots, L_j) \big) < M_f(j, L_j).
$$
By the definition of $ M_f(j, L_j) $, there exists at least a polyomino chain $ PC(L^1_3, \dots, L^1_{j-1}, L_j) $ such that
$$
M_f(j, L_j) = TI_f\big( PC(L^1_3, \dots, L^1_{j-1}, L_j) \big).
$$
Then, by Lemma~\ref{l1}, it follows that
$$
TI_f\big( PC(L_3, \dots, L_j, L_{j+1}, \dots, i) \big) < TI_f\big( PC(L^1_3, \dots, L^1_{j-1}, L_j, L_{j+1}, \dots, i) \big),
$$
which contradicts the maximality of the chain $ PC(L_3, \dots, L_{n-1}, i) $.

\item[b)] 

    Now, let us prove the  forward implication of the second statement.  Since the two polyomino chains are distinct, there must be at least one index at which they differ. Without loss of generality, assume that $j\in \{ 3, \dots, n-1\}$ is the first index (counting from right to left) such that $ L^1_j \neq L^2_j $, that is,  $L^1_k = L^2_k $ for all $ k\geq j+1$.
By Lemma~\ref{l1}, it follows that
$$
TI_f\big( PC(L^1_3, \dots, L^1_j, L^1_{j+1}) \big) = TI_f\big( PC(L^2_3, \dots, L^2_j, L^1_{j+1}) \big).
$$
Moreover, by the first statement of this lemma, for $ h = 1, 2 $,
$$
M_f(j, L^h_j) = TI_f\big( PC(L^h_3, \dots, L^h_j) \big).
$$
Therefore,
$$
M_f(j, L^1_j) + g_f(L^1_j, L^1_{j+1}) = M_f(j, L^2_j) + g_f(L^2_j, L^1_{j+1}).
$$ 
To prove the reverse implication, observe that,
$$
M_f(n, i) = TI_f\big(PC(L^1_3, \dots, L^1_{n-1}, i)\big).
$$
Then, applying the first statement and Lemma~\ref{l1}, we obtain
$$
M_f(n, i) = M_f(j+1, L^1_{j+1}) + g_f(L^1_{j+1}, L^1_{j+2}) + \dots + g_f(L^1_{n-1}, i).
$$
By hypothesis, we have
$$
M_f(j-1, 1) + g_f(1, L^1_j) = M_f(j-1, 2) + g_f(2, L^1_j),
$$
which leads to the identity,
$$
\begin{aligned}
M_f(n, i) &= M_f(j-1, 1) + g_f(1, L^1_j) + g_f(L^1_j, L^1_{j+1}) + \dots + g_f(L^1_{n-1}, i) \\
          &= M_f(j-1, 2) + g_f(2, L^1_j) + g_f(L^1_j, L^1_{j+1}) + \dots + g_f(L^1_{n-1}, i).
\end{aligned}
$$

Now, without loss of generality, assume that $L^1_{j-1} = 1$. Then, by the first statement,
$$
M_f(j-1, 1) = TI_f\big(PC(L^1_3, \dots, L^1_{j-2}, 1)\big).
$$
Moreover, there exists a polyomino chain $PC(L^2_3, \dots, L^2_{j-2}, 2)$ such that
$$
M_f(j-1, 2) = TI_f\big(PC(L^3_3, \dots, L^3_{j-2}, 2)\big).
$$

Therefore, according to Lemma \ref{l1}, the polyomino chains
$$
PC(L^1_3, \dots, L^1_{j-2}, 1, L^1_j, \dots, L^1_{n-1}, i) \quad \text{and} \quad PC(L^2_3, \dots, L^2_{j-2}, 2, L^1_j, \dots, L^1_{n-1}, i)
$$
are distinct and both attain the value $ M_f(n, i)$, as required.

\item [c)]
To prove the sufficiency direction, assume that all maximal polyomino chains associated with $M_f(n,i)$ share the same last $n - j$ links, denoted by $L_m$ for $ m \in \{j+1, \dots, n\} $, and that at least two of these chains differ at the $ j^\text{th} $ link. Selecting two such chains that differ at index $ j $, and applying the second part of this lemma, it follows that
$$
n_t(j+1, L_{j+1}) = 1 + n_t(j,1) + n_t(j,2).
$$ 

Now, to prove that $ n_t(m+1, L_{m+1}) = n_t(m, L_m) $ for all $ m \in \{n-1, \dots, j+1\} $, suppose by contradiction that
$$
n_t(m+1, L_{m+1}) = 1 + n_t(m,1) + n_t(m,2).
$$
Since every maximal polyomino chain associated with $ M_f(n,i) $ has $ L_{m+1} $ as their $(m+1)^\text{th}$ link, applying the second part of the lemma implies the existence of a maximal polyomino chain for $ M_f(n,i) $ whose $ m^\text{th} $ link differs from $ L_m $, contradicting the assumption that all maximal chains share the same last $n-j$ link.

For the reverse implication of the statement assume that  
$$
n_t(n,i) = n_t(n-1,L_{n-1}) = \dots = n_t(j+1,L_{j+1}) = 1 + n_t(j,1) + n_t(j,2)
$$  
for some $j\in \{3, \dots,n-1\} $.  From this equality chain, it is obtained that
$$M_f(n,i)=M_f(j+1,L_{j+1})+g_f(L_{j+1},L_{j+2})+\dots +g_f(L_{n-1},i).$$

Thus, we can construct at least one maximal polyomino chain for the problem $M_f(n,i)$ of the form  
$$
PC(L_3, \dots, L_j, L_{j+1}, L_{j+2}, \dots, L_{n-1}, i).
$$  
Let $A $ denote the set of all maximal polyomino chains associated with $ M_f(n,i) $. Then, every chain in $A $ must share the last $ n - j $ links with the one constructed above.  To verify this, suppose that there exists a chain in $ A $ that differs from it for the first time (from right to left) at some link of index $ k $, with $ k\in \{j+1, \dots,n-1\}$. Then, by the second part of this lemma, it follows that  
$$
n_t(k+1, L_{k+1}) = 1 + n_t(k,1) + n_t(k,2),
$$  
which contradicts our initial assumption.

Finally, we need to show that there exist at least two polyomino chains in $A$ that differ at the $j^{th}$ link. This follows directly from the fact that all chains in $A$ have their $(j+1)^{th}$ link equal to $L_{j+1}$, and since 

$$n_t(j+1, L_{j+1}) = 1 + n_t(j,1) + n_t(j,2),$$  

the second part of this lemma ensures the existence of two chains differing at the $j^{th}$ link.

Thus, both conditions required in the statement are satisfied.

\end{enumerate}

\end{proof}

\begin{thm}\label{p1}
Let $ n \geq 4 $ and $ i \in\{ 1,2\}$. Then $n_t(n,i) + 1$ equals the number of distinct maximal polyomino chains (possibly isomorphic), for the problem $M_f(n,i)$, among all polyomino chains with $ n $ squares and  final link of type $ i $. Moreover, these maximal chains can be generated by tracking the ties produced by the problem $M_f(n,i)$. 
\end{thm}

\begin{proof}
We proceed by induction on the number $ m \geq 1 $ of distinct maximal polyomino chains.
\textbf{Base case.} Let $m = 1$, and suppose that $ n_t(n,i) \geq 1 $. Then, by the recursive definition of $ n_t(n, i) $, there must exist $ j \in \{3, \dots, n-1\} $ such that
$$
n_t(n, i) = n_t(n-1, L_{n-1}) = \dots = n_t(j+1, L_{j+1}) = 1 + n_t(j, 1) + n_t(j, 2).$$ 

However, the third  statement of Lemma~\ref{l2} implies the existence of at least two distinct maximal polyomino chains, which contradicts the assumption that $ m = 1$. Therefore, the base case holds.

\textbf{Inductive step.} Assume the result holds for all $m'<m$, with $ m > 1$. By the constructive proof of the first implication in the second statement of Lemma~\ref{l2}, each pair of maximal polyomino chains is associated with some index $ j \in \{3, \dots, n-1\} $. Among the $ m $ maximal polyomino chains, there must exist a pair for which this index attains its largest value, say  $n^* < n $.  As a consequence, the maximal number of consecutive links, starting from the last one, that are shared by all maximal polyomino chains is $n - n^* $.
By the third statement of Lemma \ref{l2} it follows that $n_t(n,i)=n_t(n^*+1,L_{n^*+1})=1+n_t(n^*,1)+n_t(n^*,2)$. On the other hand, we can partition the set of maximal polyomino chains into two subsets according to the value of the link at position $ n^* $: one with $ L_{n^*} = 1 $ and another with $ L_{n^*} = 2 $, having cardinalities $ m_1 $ and $ m_2 $, respectively, where $ m_1 + m_2 = m $ and $ m_i \geq 1 $ for $ i = 1,2 $. By the first statement of Lemma~\ref{l2}, these $ m_1 $ and $ m_2 $ maximal chains can be cut into maximal polyomino chains for the subproblems $ M_f(n^*,1) $ and $ M_f(n^*,2) $, respectively. Moreover, such chains have to correspond to all maximal chains for these subproblems; otherwise, a distinct maximal chain in either subset could be extended to produce a new maximal chain for $ M_f(n,i) $, contradicting the assumption that there are exactly $ m $ maximal chains. Thus, by the induction hypothesis, we have
$$
n_t(n^*,1) = m_1 - 1 \quad \text{and} \quad n_t(n^*,2) = m_2 - 1,
$$
which implies
$$
n_t(n,i)=n_t(n^*+1,L_{n^*+1}) = 1+  (m_1 - 1) + (m_2 - 1)  = m - 1.
$$ This complete the inductive proof.

Finally, the second part of the theorem follows directly from the fact that, by keeping track of the $n_t(n,i)$ ties that arise during the recursive computation of $ M_f(n, i)$, one can construct exactly $n_t(n,i) + 1$ distinct maximal polyomino chains, which may be isomorphic.

\end{proof}

Theorem~\ref{p1} shows that, once the number of squares and the ending link are fixed, if the recursive computation of $M_f(n,i)$ yields no ties, then the constructive method produces the unique maximal polyomino chain. Conversely, if ties do occur during the recursive process, then each polyomino chain obtained by tracking these ties is maximal, and together they constitute the complete set of maximal chains under the given conditions.

At this point, it is worth emphasizing that our main objective is to determine the maximum value of a degree-based topological index among all polyomino chains with a fixed number of squares, say $ n $. Therefore, the quantity of real interest is
$$
M_f(n) := \max\{M_f(n,1), M_f(n,2)\}.
$$
In any case, the problem reduces to analyzing $ M_f(n,i) $, and thus, all the previous discussion remains fully applicable. As an initial consequence of the above results, the following proposition provides sufficient conditions under which the linear and zigzag chains are the maximal polyomino chains with respect to a degree-based topological indices.

\begin{prop}\label{p2}
Let $ TI_f $ be a degree-based topological index such that 
$$
g_f(1,1) > \max\left\{g_f(1,2),\, g_f(2,2),\, \frac{g_f(1,2)+g_f(2,1)}{2}\right\}.
$$
Then:

\begin{itemize}
    \item[a)] If $ g_f(1,1) > g_f(2) $, then among all polyomino chains with $ n \geq 3 $ squares, the linear chain $ Li_n $ uniquely maximizes $ TI_f $.
    
    \item[b)] If $ g_f(1,1) = g_f(2) $, then for all $ n \geq 4 $, $ Li_n $ uniquely maximizes $ TI_f $; and for $ n = 3 $, we have $ TI_f(PC(2)) = TI_f(PC(1)) $.
    
    \item[c)] If $ g_f(1,1) < g_f(2) $, define the threshold
   $
    n^* := \left\lceil \frac{g_f(2) - g_f(1,1)}{g_f(1,1) - g_f(2,2)} + 3 \right\rceil. 
   $
   Then:
    \begin{itemize}
       
        \item[i)] For $ 3 \leq n < n^* $,  $ Z_n $ uniquely maximizes $ TI_f $.
        
        \item[ii)] For $ n = n^* $, $ Li_n $ achieves the maximum $ TI_f $; and possibly $ Z_n $ does as well.
         \item[iii)] For $ n > n^* $, $ Li_n $ uniquely maximizes $ TI_f $.
    \end{itemize}
\end{itemize}

\end{prop}
\begin{proof}
We begin by proving the following auxiliary statement: if for some $ n \geq 4 $, 

$$
 M_f(n-1,1) + g_f(1,1) \geq M_f(n-1,2) + g_f(2,1), 
$$
then it follows that 
$$
M_f(n,1) + g_f(1,1)> M_f(n,2) + g_f(2,1).
$$
To see this, note that by definition $ M_f(n,2) $  equals either $$M_f(n-1,1) + g_f(1,2) ~ \text{or}~ M_f(n-1,2) + g_f(2,2).$$ In both cases, the assumptions
$$
2g_f(1,1) > g_f(1,2) + g_f(2,1) \quad \text{and} \quad g_f(1,1) > g_f(2,2),
$$
imply the desired inequality. 

From Remark \ref{r1} and the assumption $ g_f(1,1) > g_f(1,2)  $, it follows that, without the possibility of a tie

$$
M_f(4,1) = M_f(3,1) + g_f(1,1) > M_f(3,2) + g_f(2,1).
$$
Thus, combining the above with Theorem~ \ref{p1}, we conclude that for all $ n \geq 4 $, the value of $ M_f(n,1) $  is uniquely attained by the linear chain $ Li_n $. Now consider the global maximum $ M_f(n) $ and prove the following: if for $ n \geq 3 $, we have
$$
M_f(n,1) \geq M_f(n,2),
$$
then it follows that
$$
M_f(n+1,1) > M_f(n+1,2).
$$
The above is a direct consequence of the inequalities $ g_f(1,1) > g_f(1,2) $ and $ g_f(1,1) > g_f(2,2) $, together with the fact that $ M_f(n,1) $ for $ n \geq 4 $ is uniquely attained by $ Li_n $. As a consequence, once $ M_f(n) = M_f(n,1) $ with the possibility $ M_f(n) = M_f(n,1) = M_f(n,2) $, this equality will persist without being able to equal $ M_f(n,2) $ for larger $ n $. To verify that such an equality occurs for some number of squares, assume that there exists $ n \geq 4 $ such that $ M_f(N,2) > M_f(N,1) $ holds for all $ N\in \{4, \dots,n\} $.
 Then necessarily, without a tie
$$
M_f(N,2) = M_f(N-1,2) + g_f(2,2),
$$
otherwise, it follows that
$$M_f(N-1,2) + g_f(2,2)\leq M_f(N-1,1) + g_f(1,2) < M_f(N-1,1) + g_f(1,1)=M_f(N,1),
$$
this implies that
$$
M_f(N,2) = max \{ M_f(N-1,2) + g_f(2,2), M_f(N-1,1) + g_f(1,2) \} \leq M_f(N,1),
$$
which contradicts our assumption that $M_f(N,2)>M_f(N,1)$.

Hence, for  $N\in \{3, \dots, n\}$, the value $M_f(N,2)$ is uniquely attained by the zig-zag chain $Z_n$, and satisfies $$ M_f(n,2) = M_f(3,2) + g_f(2,2)(n - 3), $$ 
whereas we already know that $ M_f(n,1) = M_f(3,1) + g_f(1,1)(n - 3) $. 

Therefore, the condition $ M_f(n,2) > M_f(n,1) $  is equivalent to
$$
n< n^* := \left\lceil \frac{g_f(2) - g_f(1,1)}{g_f(1,1) - g_f(2,2)} + 3 \right\rceil.
$$
This completes the proof by combining the above analysis with Theorem \ref{p1}.
\end{proof}

\begin{rem}
It is worth noting that Corollary 2.11 from \cite{a1}, which includes results originally derived in \cite{a16}, is aligned with Proposition \ref{p2}.
\end{rem}

The results established in this section provide the foundation for solving concrete extremal problems. In the next section, we apply this framework to the augmented Zagreb index and fully characterize the polyomino chains that maximize it.

\section{Polyomino Chains Maximizing the AZI}\label{s2}

The \emph{augmented Zagreb index} ($ AZI $) is a molecular structure descriptor defined by Furtula, Graovac, and Vukićević in \cite{a2}. It was inspired by the success of the \emph{atom-bond connectivity} ($ ABC $) index and  is based on the function
$$
f(x,y) = \left(\frac{xy}{x + y - 2}\right)^3, x, y \in \{1, 2, \dots\}.
$$ 

The $AZI$ outperforms the $ABC$ index in predicting the heat of formation of heptanes and octanes, and has demonstrated the highest correlation with standard heats of formation and boiling points of octane isomers across multiple studies \cite{a3}. Comparative analyses with other topological indices further highlight the exceptional structural sensitivity and smoothness of the $AZI$, particularly in the context of isomeric compounds \cite{a4}. These findings confirm that $AZI$ is a highly effective and reliable descriptor for molecular property prediction.

As a direct consequence of Proposition \ref{p2} and  Corollary \ref{c1}, we obtain the following result concerning the augmented Zagreb index ($AZI$).

\begin{cor}
    Among all polyomino chains with $n$ squares, the linear chain $Li_n$ attains the minimum $AZI$ when $n \geq 6$, while the zigzag chain $Z_n$ achieves the minimum $AZI$ for $n = 3, 4, 5$.
\end{cor}

It is worth mentioning that the above result is consistent with Theorem 2.18 from \cite{a1}.  Although several extremal results regarding the $AZI$  have been reported~\cite{a5}, the specific problem of identifying the \emph{general} polyomino chain(s) with maximal $AZI$ for a fixed number of squares was posed in~\cite{a1} and, to the best of our knowledge, remains unresolved. Motivated by this, the aim of this section is to address that problem using the constructive approach introduced in Section~\ref{s3}. Before starting our main result regarding the $AZI$ index, we first introduce a notation that will simplify the characterization of the polyomino chains under consideration.

\begin{notn}
  We define a polyomino chain  with  $ m \geq 2 $ segments as an \emph{augmented zigzag of type 1}  if all its segments have length 3; this polyomino will be denoted by $ AZ^{1}_{m} $. Similarly, a polyomino chain with  $ m \geq 3 $ segments is called an \emph{augmented zigzag of type 2} if all its segments are of length 3, except for exactly one internal segment of length 2; this family will be denoted by $ AZ^{2}_{m} $. Based on this definition, the number of distinct (up to isomorphism) polyomino chains in $ AZ^{2}_{m} $ is given by $ \left\lceil \frac{m}{2} - 1 \right\rceil $. For simplicity, through this section we will use the notation $g$, $ M(n,i) $ and $ M(n) $ instead of $g_f$, $ M_f(n,i) $ and $ M_f(n) $, respectively. 
\end{notn}

\begin{thm}\label{t2}
For $ n \geq 5 $, the polyomino chains that maximize the $ AZI $ index are as follows: $ AZ^{1}_{\frac{n-1}{2}} $ when $ n $ is odd and $ AZ^{2}_{\frac{n}{2}} $ when $ n $ is even. Specifically,

 $$M(n)= \frac{4456}{125}n-\frac{26763}{2000}-\frac{2312}{3375}I_{\{n=2k\}}.$$

\end{thm}

\begin{proof}
According to Theorem~\ref{t1}, the first step is to determine $M(4)$. Observe that
$$
M(3,1) + g(1,1) < M(3,2) + g(2,1), \quad \text{since } 2g(1,1) < g(2) + g(2,1),
$$  
and  
$$
M(3,1) + g(1,2) > M(3,2) + g(2,2), \quad \text{since } g(1,1) + g(1,2) > g(2) + g(2,2).
$$  From these inequalities, we can determine the values of $ M(4,1) $ and $ M(4,2) $. Then, by Remark \ref{r1}, it follows that,
$
M(4)=M(4,1) = M(4,2)$, i.e. $M(4,1)- M(4,2)=0$.

At this point, it is no longer necessary to compute the full values of $ M(n, i) $ to determine $ M(n+1, i) $. It is sufficient to analyze the differences between them, which simplifies the process and allows us to determine the optimal construction using only the function $ g $. 

For $ n = 5 $, the inequalities $g(1,1) < g(2,1)$ and $g(1,2) < g(2,2)$, imply that
$$
M(5,1) = M(4,2) + g(2,1) > M(4,1) + g(1,1), 
$$  
$$
M(5,2) = M(4,2) + g(2,2) > M(4,1) + g(1,2).
$$  
Therefore,  
$$
M(5) = M(5,1) = M(4,2) + g(2,1),
$$  
because $ g(2,1) > g(2,2) $. Moreover, the difference is given by $M(5,1)- M(5,2)=g(2,1)-g(2,2)$. Now, for $ n = 6 $, a similar argument yields:  
$$
M(6,1) = M(5,2) + g(2,1) > M(5,1) + g(1,1),
$$  
$$
M(6,2) = M(5,1) + g(1,2) > M(5,2) + g(2,2).
$$  
Hence,  
$$
M(6) = M(6,1)=M(5,2) + g(2,1),
$$  
and $M(6,1)- M(6,2)=g(2,2)-g(1,2)$. Next, for $ n = 7 $, we observe that:  
$$
M(7,1) = M(6,2) + g(2,1) > M(6,1) + g(1,1), \quad \text{since } g(2,2) + g(1,1) < g(2,1) + g(1,2),
$$  
$$
M(7,2) = M(6,1) + g(1,2) = M(6,2) + g(2,2),
$$  
that is, the computation of $M(7,2)$ yields a tie. Consequently,  
$$
M(7) = M(7,1)= M(6,2) + g(2,1) ,
$$  
since $ g(2,1) > g(2,2) $, and $M(7,1)- M(7,2)=g(2,1)-g(2,2)$. Before proceeding further, we summarize the key information obtained so far in Figure~\ref{f1}, as follows:
\begin{itemize}
    \item[a)] The rows represent the number of squares, while the columns correspond to the ending links.
    \item[b)] The underlined cell(s) in the $ i $-th row indicate(s) the ending link(s) of the maximal polyomino chains consisting of $ i $ squares.

    \item[c)] Green arrows originating from the cell at the $i$-th row and $j$-th column and pointing to the cell at the $(i-1)$-th row and $p$-th column indicate that $p$ is a possible value of the link at step $i-1$ in a maximal polyomino chain with $i$ squares and $L_i = j$.

    \item[d)] The expression in the cell corresponding to the $ i $-th row and the first column represents the terms in $ M(i,1) $ that do not appear in $ M(i,2) $, as determined by the recursive application of Lemma~\ref{l1}. Similarly, the second column contains the terms in $ M(i,2) $ that are not present in $ M(i,1) $. These terms can be identified by the difference $ M(i,1) - M(i,2) $.

\end{itemize} According to the key information summarized in Figure~\ref{f1}, it is evident that for $ n = 7 $, we arrive at the same scenario as for $ n = 5 $. This indicates the onset of a cyclic pattern in the decision process, with a period of two.  Therefore, based on the previous information and Theorem~~\ref{p1}, we conclude the following:  
When $ n $ is odd, the polyomino chain that maximizes $ AZI $ is defined by the sequence of links  
$$
(1,\underbrace{2,1,\dots}_{\frac{n-3}{2} \text{ times}}).
$$  
This corresponds to a polyomino chain composed of $ \frac{n-1}{2} $ segments, each of length 3. In other words, it is the polyomino $ AZ^{1}_{\frac{n-1}{2}} $. When $ n $ is even,  following the decision flow and
taking into account tie cases, the polyomino chains that maximize $ AZI $ are determined by the following sequences of links:
$$
(1,\underbrace{2,1,\dots}_{\frac{n-6}{2} \, \text{times}},2,2,1)=(\underbrace{1,2,\dots}_{\frac{n-4}{2} \, \text{times}},2,1),
$$ and
$$
(L_3,\dots,L_{n-3},1,2,1),
$$
where $M(n-2,1)=TI_f(PC(L_3,\dots,L_{n-3},1)) $  with $ n-2 \geq 6 $. Specifically, for $ n = 2k $, we can prove by induction on $ k $ that all such maximizing chains follow link sequences of the form:
$$
(\underbrace{1,2,\dots}_{i \, \text{times}},\underbrace{2,1,\dots}_{\frac{n-2}{2} - i \, \text{times}}),
$$
for $ i\in \{1, \dots, \frac{n-4}{2} \}$. To establish the inductive step, we first verify the base case. For $ n = 8 $, the polyomino chains that maximize $ AZI $ are determined by the following sequences of links:
$$
(1,2,2,1,2,1)\quad  \text{and} \quad(1,2,1,2,2,1). 
$$ Now, we assume that the claim holds for some $ k $, and let us prove it for $ k+1 $. For $ n = 2k+2 $, we have the following sequences of links:

$$
(1,\underbrace{2,1,\dots}_{\frac{n-6}{2} \, \text{times}},2,2,1),
$$
which corresponds to the case $ i = \frac{n-4}{2} $ and
$$
(\underbrace{1,2,\dots}_{i \, \text{times}},\underbrace{2,1,\dots}_{\frac{n-4}{2} - i \, \text{times}},2,1)
= (\underbrace{1,2,\dots}_{i \, \text{times}},\underbrace{2,1,\dots}_{\frac{n-2}{2} - i \, \text{times}}),
$$
for $ i\in\{1, \dots, \frac{n-6}{2} \}$, by the induction hypothesis. This completes the proof of the stated form. In terms of segments, these sequences correspond to polyomino chains with $ \frac{n}{2} $ segments, all of length 3 except for one internal segment of length 2. In other words, they belong to the family $ AZ^{2}_{\frac{n}{2}} $. Finally, the concrete computations of $M(n)$ follows directly from Lemma~\ref{l1}.

\begin{figure}[ht]
\begin{center}
\begin{tikzpicture}[
  every node/.style={minimum size=1cm, anchor=center},
  arrow/.style={-stealth, thick, green!70!black},
  underline/.style={-stealth, thick,red!70!black},
  every label/.style={font=\small}
]

\matrix (m) [matrix of nodes, row sep=0.9cm, column sep=2.5cm] {
  \textbf{$n \backslash L_n$} & \textbf{1} & \textbf{2} \\
  \textbf{3} &  \reduline{$g(1,1)$} & $g(2)$ \\
  \textbf{4} & \reduline{0} &  \reduline{0} \\
  \textbf{5} & \reduline{$g(2,1)$} & $g(2,2)$ \\
  \textbf{6} & \reduline{$g(2,2)$} & $g(1,2)$ \\
  \textbf{7} & \reduline{$g(2,1)$} & $g(2,2)$ \\
};

\draw[arrow] (m-3-3) -- (m-2-2); 
\draw[arrow] (m-3-2) -- (m-2-3); 

\draw[arrow] (m-4-2) -- (m-3-3); 
\draw[arrow] (m-4-3) -- (m-3-3); 

\draw[arrow] (m-5-3) -- (m-4-2); 
\draw[arrow] (m-5-2) -- (m-4-3); 

\draw[arrow] (m-6-3) -- (m-5-2); 
\draw[arrow] (m-6-2) -- (m-5-3); 
\draw[arrow] (m-6-2) -- (m-5-3); 
\draw[arrow] (m-6-3) -- (m-5-3); 

\end{tikzpicture}
\end{center}
\caption{Graphical representation of the behavior observed in Theorem~\ref{t2}.} 
\label{f1}
\end{figure}
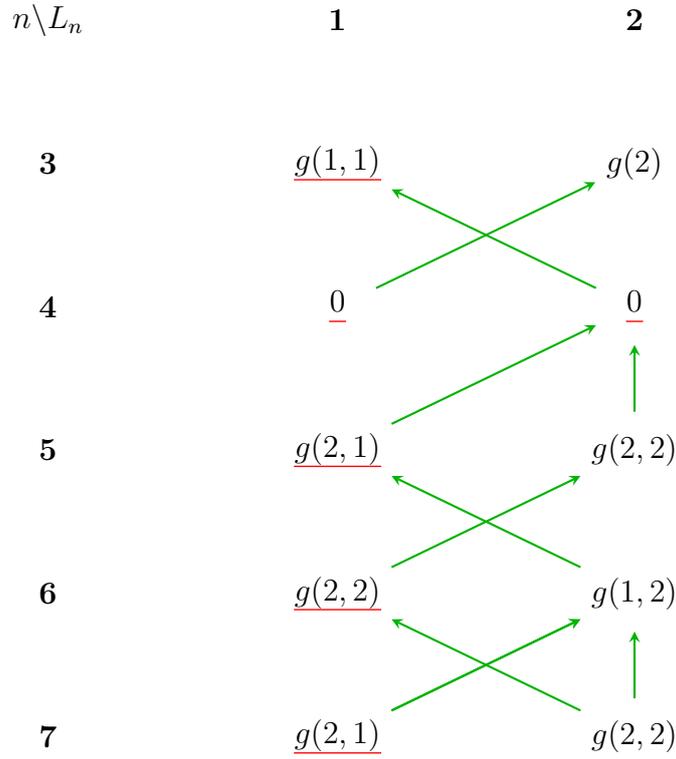

\end{proof}

\begin{rem}
As a direct consequence of the previous proof, one can verify that for $ n = 3 $, the polyomino chain that maximizes $ AZI $ is $ Li_3 $. For $ n = 4 $, the polyomino chains that maximize the $ AZI $ index are $ PC(1,2) $ and $ PC(2,1) $, both of which are isomorphic to a chain composed of two segments of lengths 3 and 2. For $ n \geq 5 $, we have that $ M(n) = M(n,1) > M(n,2) $. Moreover, when $ n $ is odd, $ M(n,1) $ has no ties; when $ n $ is even, it has exactly $ \frac{n - 6}{2} $ ties. In accordance with Theorem~ \ref{p1}, the proof yields exactly one maximal polyomino chain when $ n $ is odd, and $ \frac{n - 6}{2} + 1 $ distinct maximal polyomino chains when $ n $ is even, possibly isomorphic. Furthermore, the number of distinct polyomino chains with $ n $ squares, when $ n $ is even,   that attain the maximum $ AZI $, up to isomorphism, is given by $\left\lceil \frac{n}{4} - 1 \right\rceil.$

\end{rem}

\section{Final Discussion and Code}\label{s5}

In this paper, we have introduced a general framework for identifying extremal polyomino chains with respect to degree-based topological indices, based on dynamic programming approach. As a result, we have established extremal findings concerning linear and zigzag chains.  In particular, we have fully characterized the polyomino chains that maximizes the $AZI$ among all polyomino chains with a fixed number of squares. Despite these advances, many questions remain concerning extremal (\emph{general}) polyomino chains, both for the same and for other topological indices. To the best of our knowledge, these include the general Randi\'c index $R_{\gamma}$ (for $\gamma \leq 0$) and the general atom-bond connectivity index $ABC_{\gamma}$ (for $\gamma \leq 1$). We believe that the proposed methodology offers a valuable tool for addressing these and other related challenges in the study of extremal structures under topological descriptors.\\

Finally, the following link gives access to our implementation of the constructive method, which efficiently computes, for any degree-based topological index and any specified number of squares $n$, a polyomino chain with either maximum or minimum value in linear time with respect to $n$. Moreover, it enables the computation of the complete set of such extremal polyomino chains in linear time with respect to the cardinality of such set. The code also includes our implementation of the sufficient condition for linear and zig-zag polyomino chains as extremal chains, along with detailed comments explaining each method: \href{https://colab.research.google.com/drive/1Rntmks8DTUWgDPridrd71tpCI8n6u_yM?usp=sharing}{Link to the Code}.

\subsection*{Acknowledgment}

\section*{Funding Information}
M. Montes-y-Morales and H. Cruz-Suárez received support from VIEP through grant VIEP-00544-2025. Saylé Sigarreta was supported by CONAHCYT 2023-2024 project CBF2023-2024-1842.

\bibliography{Bib}
\bibliographystyle{acm}

\end{document}